\documentclass[12pt]{amsart}
\usepackage[english]{babel}
\usepackage{amsmath,amsthm,amsfonts,amssymb,epsfig,color, vector}
\usepackage[left=1in,top=1in,right=1in]{geometry}

\newtheorem{thm}{Theorem}
\newtheorem{lem}[thm]{Lemma}
\newtheorem{prop}[thm]{Proposition}
\newtheorem{cor}[thm]{Corollary}

\newtheorem*{rem}{Remark}

\newcommand{\E}{\mathbb{E}}

\newcommand{\prob}{\mathbb{P}}

\newcommand{\R}{\mathbb{R}}

\newcommand{\N}{\mathbb{N}}
\newcommand{\F}{\mathcal{F}}

\title[REM in a random magnetic field]{Microcanonical analysis of the Random Energy Model \\ in a random magnetic field}

\author[L.-P. Arguin]{Louis-Pierre  Arguin}            
 \address{L.-P. Arguin\\ 
 D\'epartement de Math\'ematiques et Statistique\\
 Universit\'{e} de Montr\'{e}al\\ 2920 chemin de la Tour\\
Montr\'{e}al, QC H3T 1J4 \\ 
Canada}
\thanks{L.-P. A. is supported by a NSERC discovery grant and a grant FQRNT {\it Nouveaux chercheurs}.}
\email{arguinlp@dms.umontreal.ca}

\author[N. Kistler]{Nicola Kistler}
\address{N. Kistler\\
Department of Mathematics\\
College of Staten Island \\
City University of New York\\
2800 Victory Boulevard \\
Staten Island 10314 New York
}
\email{nicola.kistler@csi.cuny.edu}

\keywords{Spin Glasses, Random Energy Model, Extremal Processes} \subjclass[2010]{Primary: 60G15, 82B44}

\date{10 February, 2014}

\begin{document}

\maketitle

\begin{abstract}
We study the spin glass system consisting of a Random Energy Model coupled with a random magnetic field.
This system was investigated by de Oliveira Filho, da Costa and Yokoi {\it(Phys. Rev. E 74 [2006])} who computed the free energy. 
In this paper, we recover their result rigorously using elementary large deviations arguments and a conditional second moment method.
Our analysis extends at the level of fluctuations of the ground states. 
In particular, we prove that the joint distribution of the extremal energies has the law of a Poisson process with exponential density after a 
recentering, which is random as opposed to the standard REM. 
One consequence is that the Gibbs measure of the model exhibits a one-step replica symmetry breaking as argued by de Oliveira Filho {\it et al.} using the replica method.

\end{abstract}

\section{Introduction and main results}
We consider the Hamiltonian of a disordered system composed of a Random Energy Model (REM) Hamiltonian coupled with a random magnetic field, that is,
\begin{equation}
\label{eqn: H}
H_N(\sigma)= X_N(\sigma) + Y_N(\sigma) \qquad \sigma\in \Sigma_N:= \{-1,+1\}^N
\end{equation}
where $X_N=(X_N(\sigma), \sigma\in \Sigma_N)$ are IID centered Gaussians of variance $N$, and
\begin{equation}
Y_N(\sigma):=\sum_{i=1}^N h_i \sigma_i
\end{equation}
where $\vec{h}=(h_i, i\in\N)$ are IID random variables independent of $X_N$. 
The random variables $X_N$ and $\vec{h}$ are defined on a probability space $(\Omega, \F,\prob)$ with the expectation denoted by $\E$. We will assume that  $\E[\exp th_1]<\infty$ for all $t\in \R$.

The motivation to study this spin glass model is two-fold. First, we are interested in understanding the effect of a random magnetic field on a glassy transition from a rigorous standpoint.
This effect is well understood in the case of ferromagnetic models, see for example \cite{aizenman-wehr}, but few rigorous results are known in the literature for spin glasses to our knowledge.
Second, the model gives a non-trivial example of a solvable spin glass where the (by now) standard approaches such as Talagrand's cavity method \cite{talagrand} or Guerra's interpolation scheme \cite{guerra} cannot be applied. To overcome this obstacle, we resort to the classical tools of large deviations, and to a {\it conditional second moment method}. This approach allows to derive a complete picture of the phase transition up to the level of the fluctuations of the ground states. The model may thus shed some light on the connections at the microcanonical level between the standard treatment of statistical mechanics models {\it \`a la Gibbs} and the successful Parisi approach based on the ultrametric structure of the Gibbs measure, see \cite{panchenko} for a review of the rigorous results in mean-field models.

The first result of the paper is the computation of the free energy.
\begin{thm}
\label{thm: free}
For $\beta>0$,
$$
\lim_{N\to\infty} \frac{1}{N}\log \sum_{\sigma\in \Sigma_N}\exp \beta H_N(\sigma)
=
\begin{cases}
\log 2 +\frac{\beta^2}{2} + \E[ \log \cosh \beta h_1] \ &\text{ if $\beta\leq \beta_c$}\\
E_{\max}\beta \ &\text{ if $\beta\geq \beta_c$} 
\end{cases}
\ \text{ $\prob$-a.s.}
$$
where $\beta_c$ satisfies the self-consistency equation
\begin{equation}
\label{eqn: beta c}
\beta_c=\sqrt{2\big(\log 2 - (\beta_c \E[h_1\tanh \beta_ch_1] -\E[\log\cosh \beta_c h_1])\big)}
\end{equation}
and
\begin{equation}
\label{eqn: E_max}
E_{\max}=\beta_c+ \E[h_1\tanh\beta_c h_1]\ .
\end{equation}
\end{thm}

\begin{rem}
By concentration of measure, one easily sees that the same result holds for the $\E$-average of the free energy.
\end{rem}

Equation \eqref{eqn: E_max} gives the partition of the maximal energy density between the one of the REM, $\beta_c$, and the one of the random field interaction,  $\E[h_1\tanh \beta_ch_1]$.
The formula for the free energy was first obtained by de Oliveira Filho, da Costa and Yokoi \cite{filho-dacosta-yokoi}. 
Our contribution is to rigorously back their argument using large deviation techniques.
The idea is that the random field energy density
\begin{equation}
 y_{N,\vec h}(\sigma):=\frac{1}{N} \sum_ih_i \sigma_i 
\end{equation}
satisfies a large deviation principle (LDP) with rate function $I$ conditionally on $\vec h$ by the G\"artner-Ellis theorem, cf. Lemma \ref{lem: LDP}.
Thus for a given energy $E$, if the typical value of $ y_{N,\vec h}$ is $y^*=y^*(E)$, then there are approximately $\exp(N(\log 2 -I(y^*)))$ $\sigma$'s for which $y_{N,\vec{h}}$ is approximately $y^*$.
Since the random field $Y_N(\sigma)$ is independent of the REM Hamiltonian $X_N(\sigma)$, the system essentially reduces to a REM model on $\exp(N(\log 2 -I(y^*)))$ Gaussian variables of variance $N$. 
In particular, the freezing of the model occurs not only because $\beta$ is increased, but also because the number of relevant configurations decreases as $y^*$ increases with $\beta$.
From the expression of the critical $\beta$ for the REM, we thus expect the system to freeze at $\beta_c=\sqrt{2(\log 2-I(y^*(E_{\max}))}$. 

It is interesting to remark that tight upper bounds for the free energy can be obtained by means of fractional moments, and annealing. To see this, let $m\in(0,1]$. It then holds: 
\begin{equation} \begin{aligned}
\frac{1}{N} \log \sum_{\sigma \in \Sigma_N} \exp \beta H_N(\sigma) & = 
\frac{1}{N m} \log \left(\sum_{\sigma \in \Sigma_N} \exp \beta H_N(\sigma)\right)^m \\
& \leq \frac{1}{N m} \log \sum_{\sigma \in \Sigma_N} \exp m \beta H_N(\sigma) \\
& =  \frac{1}{N m} \log \sum_{\sigma \in \Sigma_N} \exp\left[  m \beta X_N(\sigma) + \beta m Y_N(\sigma)\right]\,, \label{frac_one}
\end{aligned}\end{equation}
the second step by straightforward convexity arguments. Consider now the Gibbs measure 
\begin{equation}
P_{\beta, m, \boldsymbol h}(\sigma) = \frac{1}{Z_N(\beta, m, \boldsymbol h)}\exp \beta m Y_N(\sigma),
\end{equation}
for $Z_N(\beta, m, \boldsymbol h) = \exp\left[N \log 2+ \sum_{i=1}^N \log \cosh(\beta m h_i) \right]$. Denoting by $E_{\beta, m, \boldsymbol h}$ the (quenched) expectation with respect to this tilted coin tossing measure, we may reformulate the right-hand side of \eqref{frac_one} 
to obtain 
\begin{equation} \begin{aligned}
& \frac{1}{N} \log \sum_{\sigma \in \Sigma_N} \exp \beta H_N(\sigma) \leq \\
& \qquad \leq 
\frac{1}{Nm} \log E_{\beta, m, \boldsymbol h}\left[ \exp \beta m X_N(\sigma)\right] + \frac{\log 2}{m} + \frac{1}{N m} \sum_{i=1}^N \log \cosh(\beta m h_i).
\end{aligned} \end{equation}
Taking now expectation with respect to the environment, and using Jensen's inequality (the annealing) yields
\begin{equation} \begin{aligned}
& \E\left[ \frac{1}{N} \log \sum_{\sigma \in \Sigma_N} \exp \beta H_N(\sigma) \right] \leq \frac{\beta^2 m}{2} + \frac{\log 2}{m} + \frac{1}{m} \E \log \cosh(\beta m h_1).
\end{aligned} \end{equation}
As this holds for {\it any} $m\in (0,1]$, we have in fact the following upper bound for the 
free energy: 
\begin{equation}
\E\left[ \frac{1}{N} \log \sum_{\sigma \in \Sigma_N} \exp \beta H_N(\sigma) \right] 
\leq \inf_{0 < m \leq 1} \left\{\frac{\beta^2 m}{2} + \frac{\log 2}{m} + \frac{1}{m} \E \log \cosh(\beta m h_1)\right\}.
\end{equation}
The variational principle on the right-hand side can be easily solved. Omitting the elementary considerations, 
one indeed recovers the limiting free energy as established in Theorem \ref{thm: free}. Of course, 
this provides an upper bound only, but the simplicity of the method is somewhat puzzling. We are not aware of a similarly efficient method to derive matching lower bounds. \\

The second result of the paper addresses the fluctuations of the extremal energies.
\begin{thm}
\label{thm: extremal}
There exists $c_1=c_1(N,\vec{h})$ and $c_2=c_2(N,\vec h)$ such that for 
\begin{equation}
\label{eqn: r_N}
r(N,\vec{h}):= c_1 N -c_2\ \log N\ ,
\end{equation}
the point process 
 $(H_{N}(\sigma)- r(N,\vec{h}), \sigma\in \Sigma_N)$ conditioned on $\vec{h}$ converges weakly to a Poisson process with intensity $C e^{-\beta_c z} dz$ 
 for some explicit deterministic constant $C>0$ and for $\prob$-almost all $\vec{h}$.
 Moreover, 
\begin{equation}
\label{eqn: conv c}
c_1(N,\vec{h}) \to E_{\max} \qquad c_2(N,\vec h)\to \frac{1}{2\beta_c}\qquad \text{$\prob$-a.s.}
\end{equation}
\end{thm}
It follows from the theorem that the maximum of $H_N(\sigma)$ given $\vec{h}$ has Gumbel fluctuations
around the recentering term $r(N,\vec{h})$.
We stress that the result is {\it quenched} in the sense that the convergence holds given the random field $\vec{h}$ (thus also when the random field is averaged). 
Nevertheless, the fluctuations (perhaps surprisingly) remain independent of the realization of the field.
We remark that the theorem includes the particular case of the REM model with deterministic field studied by Bovier \& Klimovsky \cite{bovier-klimovsky}.
The proof in  \cite{bovier-klimovsky} relies on the precise knowledge from elementary combinatorics of the number of spin configurations with a given magnetization. 
This is impossible to do in the case of $y_{N,h}$ because $\vec{h}$ is random. 
Instead, we generalize the argument by noticing that proving a central limit theorem for this order parameter suffices, cf. Lemma \ref{lem: CLT}.
Finally, we point out that a frontal attack with the recentering term $r_N=E_{\max} N - \frac{1}{2\beta_c} \log N$ as
suggested by  \eqref{eqn: conv c}  will fail. This choice only works in case of a {\it deterministic} magnetic field. In our case, the fluctuations of $y_{N,\vec{h}}(\sigma)$ are too large to ensure convergence. As it turns out, the choice of $c_1$ must depend on the randomness in such a way that the fluctuations are on the right scale for convergence. This delicate point will be emphasized in the proof. \\

As a consequence of Theorem \ref{thm: extremal}, we obtain the law of the Gibbs measure at low temperature and the overlap distribution.
For this, we denote the normalized Gibbs weight by
$$
G_{\beta,N}(\sigma):= \frac{e^{\beta H_N(\sigma)}}{Z_N(\beta)} 
$$
The overlap between $\sigma,\sigma'\in \Sigma_N$ is defined as $R_N(\sigma,\sigma'):= \frac{1}{N}\sum_{i=1}\sigma_i\sigma_i'$.
The form of the overlap distribution was obtained by de Oliveira Filho, da Costa and Yokoi using the replica method.
\begin{thm}
\label{thm: gibbs}
For $\beta>\beta_c$, the normalized Gibbs weights
$
\left( G_{\beta, N}, \sigma\in \Sigma_N\right)_\downarrow
$
 ordered in decreasing order
converges to a Poisson-Dirichlet variable with parameter $\beta_c/\beta$ as $N\to\infty$. 
\end{thm}
\begin{cor}
\label{cor: overlap}
The two-overlap distribution converges in law to a sum of two delta masses:
\begin{equation}
\label{eqn: overlap}
\E G_{\beta,N}^{\times 2}\{ R_N(\sigma,\sigma')\in dq\} \to \frac{\beta_c}{\beta} \delta_{q} + \left(1-\frac{\beta_c}{\beta}\right)\delta_1
\end{equation}
where $q=\E[\tanh^2 \beta_c h_1]$.
\end{cor}
In physics terms, one can interpret the individual $\sigma$'s with extremal energies, i.e.\@ close to $r_N$, as the {\it pure states}. 
Their Gibbs weight is macroscopic with Poisson-Dirichlet distributions. 
In the course of the proof, we will show that these optimal $\sigma$'s are chosen among the configurations with order parameters $y_{N,h}(\sigma)\approx y^*(E_{\max})$. It is to be noted that the overlap is strictly non-zero even when the random field $\vec{h}$ is centered.
In that case, the pure states exhibit a zero magnetization but the satisfaction of the random field constraint creates a non-zero overlap between them.
\\

Throughout the paper, the notation $o(1)$ denotes a term that goes to $0$ when $N\to\infty$.
The uniform measure on the hypercube $\Sigma_N$ will be denoted by $\mu_N$. The expectation of a function $f$ on $\sigma_N$ with respect to $\mu_N$ will sometimes be denoted $\mu_N(f)$ for short.

\section{Proof of Theorem \ref{thm: free}}
The first step to compute the free energy is to obtain the entropy of configurations at a given energy level, cf. Proposition \ref{prop: entropy}. 
For this purpose, the entropy for the random field energy density $y_{N,\vec{h}}(\sigma)$ is needed.
Note that, by the strong law of large numbers, the typical value under the uniform measure is
$$
y_{N,\vec{h}}(\sigma)=\frac{1}{N}\sum_{i=1}^N h_i\sigma_i \to 0 \ \  \text{ $\prob\times \mu_N$-a.s.}
$$left-hand side
Moreover, by taking $\sigma_i=\text{sgn} \ h_i$, it follows that $y_{N,\vec{h}}(\sigma) \leq \E|h_1|$ for all $\sigma$ for $N$ large enough. 
It turns out that we can also control the large deviations of $y_{N,\vec{h}}$ around the mean 
under $\mu_N$ for $\prob$-almost all realizations of $\vec{h}$.
\begin{lem}
\label{lem: LDP}
On a set of $\vec{h}$ of $\prob$-probability one, 
the variables $(y_{N,\vec{h}})_{N\in\N}$ on the probability space $(\Sigma_N, \mu_N)$ satisfies an LDP with
rate function
\begin{equation}
\label{eqn: I}
I(y):= \sup_{t\in\R} \left\{yt - \psi(t) \right\} \qquad \text{for $y \in [-\E|h_1|, \E|h_1|]$}
\end{equation}
where $\psi(t):= \E[\log\cosh th_1]$. 
In particular, if $F:\R\to\R$ is a continuous function that is bounded above
\begin{equation}
\label{eqn: varadhan}
\lim_{N\to\infty} \frac{1}{N} \log \mu_N\left( e^{ NF(y_{N,\vec{h}}(\sigma))} \right)=\sup_{y\in \R} \{F(y)-I(y)\}\ .
\end{equation}
\end{lem}

\begin{proof}
The second assertion follows from the first by Varadhan's Lemma, see e.g. \cite{denhollander}.
For the first, the G\"artier-Ellis theorem, see also \cite{denhollander}, guarantees a LDP for  $(y_{N,\vec{h}})_{N\in\N}$ under $\mu_N$ if
$$
\lim_{N\to\infty} \frac{1}{N}\log \mu_N\big(\exp tNy_{N,\vec{h}}(\sigma)\big)\to \psi(t) \ \text{ for all $t$ on a set of $\vec{h}$ of probability one}\ .
$$
For a given $t\in\R$, this is easy since the left-hand side equals  $\frac{1}{N} \sum_{i=1}^N \log \cosh th_i$, 
and the convergence follows by the strong law of large numbers. To extend the convergence for all $t$ of  on a set of $\vec{h}$ of probability one, 
it suffices to consider a countable dense set $(t_k)_k$'s. 
Since the derivative of $\psi(t)$ is bounded uniformly, $\psi(t)$ can be approximated by $\psi(t_k)$ uniformly which yields the convergence for all $t$. 
\end{proof}

For $E\in\R$, consider the entropy, i.e., the log-number of configurations with energy density in a small interval around $E$:
\begin{equation}
\begin{aligned}
S_N(E)&:= \frac{1}{N} \log  \#\big\{\sigma\in \Sigma_N: H_N(\sigma)\in [E N, EN+ \sqrt{N}] \big\}\ .
\end{aligned}
\end{equation}
(The width $\sqrt{N}$ is an educated choice to match the order of the fluctuations of $H_N$.)
This quantity is random for finite $N$. Interestingly, it self-averages given $\vec h$ as the next proposition shows. 
For the purpose of the statement, let
\begin{equation}
\label{eqn: S}
S(E,y):= \log 2 - \frac{1}{2}(E-y)^2 -I(y)\qquad S(E):= \max_{ -\E|h_1|\leq y\leq \E|h_1|} S(E,y)\ ,
\end{equation}
as well as $E_{\max}:=\sup \{E\in \R: S(E)>0\}$ and $E_{\min}:=\inf \{E\in \R: S(E)>0\}$.
By continuity, $S(E_{\max})=S(E_{\min})=0$.
\begin{prop}
\label{prop: entropy}
For $E \in (E_{\min}, E_{\max})$, there exists $C>0$ (independent of $E$ and $\vec h$) such that 
$$
\prob\Big(|S_N(E) - S(E)|>\delta \ \Big | \vec{h} \Big)\leq e^{-CN} \qquad \text{$\forall \delta>0$.}
$$
where $\prob(\cdot\ | \vec{h})$ is the conditional probability given $\vec{h}$.
\end{prop}
\begin{proof}
For convenience, define $\mathcal N_N(E):=  \#\big\{\sigma\in \Sigma_N: H_N(\sigma)\in [E N, E N+\sqrt{N}]\}$.
The proof is split in two steps. 
First we show that
\begin{equation}
\label{eqn: E N}
\frac{1}{N}\log \E[\mathcal N_N(E)| \vec{h}] \to S(E) \text{ \ $\prob$-a.s.}
\end{equation}
Since $X_n$ is Gaussian of variance $N$, the left-hand side is
$$
 \log 2 +  \frac{1}{N}\log \mu_N\left(\int_{[E N, E N+\sqrt{N}] -N y_{N,\vec{h}}(\sigma)} \frac{e^{-z^2/2N}}{\sqrt{2\pi N}} dz\right)
$$
which by the change of variables $u=\frac{1}{\sqrt{N}}\big(z-N(E-y_{N,\vec{h}}(\sigma))\big)$ equals
\begin{equation}
\label{eqn: log number}
\log 2 +  \frac{1}{N}\log \mu_N\left( e^{-\frac{N(E-y_{N,\vec{h}}(\sigma))^2}{2}} 
\int_{[0,1] } \frac{e^{-u^2/2} e^{-\sqrt{N}u(E-y_{N,\vec{h}}(\sigma))}}{\sqrt{2\pi}} du\right)
\end{equation}
We apply Equation \eqref{eqn: varadhan} with $F(y)=-\frac{(E-y)^2}{2}$. 
The integral is at most of order $e^{O(\sqrt{N})}$ hence will not contribute. This gives $S(E)$ and proves \eqref{eqn: E N}.

Second, using  {\it the  second moment method with conditioning}, we show

\begin{equation}
\label{eqn 2nd moment}
\lim_{N\to \infty} \left| \frac{\mathcal N_N(E)}{\E[\mathcal N_N(E)| \vec{h}]} -1 \right|=0 \text{ in $\prob( \ |\vec{h})$-probability}
\end{equation}
with exponential decay, which is sufficient for our purpose. By Markov's inequality, it suffices to show
\begin{equation}
\label{eqn: decay}
\frac{\E[\mathcal N^2_N(E)| \vec{h}] -  \E[\mathcal N_N(E)| \vec{h}]^2}{ \E[\mathcal N_N(E)| \vec{h}]^2}\to 0\ \text{, exponentially fast.}
\end{equation}
But
$$
\E[\mathcal N^2_N(E)| \vec{h}]= \E[\mathcal N_N(E)| \vec{h}]+ \E[\mathcal N_N(E)| \vec{h}]^2 
- \sum_{\sigma\in \Sigma_N} \prob\big(X_N(\sigma)/N \in [E, E+1/\sqrt{N}]-y_{N,\vec h}(\sigma)\ \big| \vec{h}\big)^2\ .
$$
The last term is $o(1)  \E[\mathcal N_N(E)| \vec{h}]$, thus \eqref{eqn: decay} holds if $\E[\mathcal N_N(E)| \vec{h}]\geq e^{CN}$ for some $C>0$. 
But this is ensured by \eqref{eqn: E N} for $E\in (E_{\min}, E_{\max})$. 
\end{proof}

The following result is a Gibbs variational principle for the free energy which follows from a standard application of Laplace's method. 
The control that is needed around the extremal energies is in the spirit of what is needed in the proof of Theorem \ref{thm: extremal}.
\begin{prop}
\label{prop: Gibbs}
For $\beta>0$, 
$$
\lim_{N\to\infty} \frac{1}{N} \log \sum_{\sigma\in \Sigma_N}\exp \beta H_N(\sigma)
= \max_{E\in [E_{\min}, E_{\max}]} \big\{ \beta E + S(E)\big\} \ \text{ $\prob$-a.s. }
$$
\end{prop}
\begin{proof}
Let $f_N(\beta):=  \frac{1}{N} \log \sum_{\sigma\in \Sigma_N}\exp \beta H_N(\sigma)$ and
$f(\beta)= \max_{E\in [E_{\min}, E_{\max}]} \big\{ \beta E + S(E)\big\}$. 
We will show that for $\delta>0$, 
$$
\prob(|f_N(\beta)- f(\beta)|> \delta) \to 0
$$
exponentially fast in $N$ which is sufficient for almost sure convergence. 
We show $\lim_{N\to\infty} \prob(f_N(\beta)> f(\beta)+ \delta) \to 0$.
The other bound $\prob(f_N(\beta)< f(\beta)-\delta)$ is done similarly. 

We first establish a control on the entropy of the energy levels using the previous results.
The energy levels around $E_{\min}$ and $E_{\max}$ will be treated more carefully.
For $\varepsilon>0$, consider the event
$$
A_\varepsilon=\big\{\forall \sigma\in \Sigma_N, \ H_N(\sigma) \in  [E_{\min} -\varepsilon, E_{\max} +\varepsilon] \big\}\ .
$$
We note that $\prob(A_\varepsilon^c |\vec{h})$ converges to $0$ exponentially fast by Markov's inequality and \eqref{eqn: E N}.
Consider also for $\rho>0$.
$$
B^{\max}_{\varepsilon,\rho}=\big\{ \#\{\sigma\in \Sigma_N: H_N(\sigma)\in (E_{\max}-\varepsilon, E_{\max}+\varepsilon)\}\leq e^{\rho N}\big\}
$$
and consider $B^{\min}_{\varepsilon,\rho}$ defined similarly with $E_{\min}$. 
We show $\prob(B_{\varepsilon,\rho}^{\max}|\vec{h})\to 1$ exponentially fast. 
By Markov's inequality,
$$
\prob((B_{\varepsilon,\rho}^{\max})^c|\vec{h})\leq e^{-\rho N} \E\big[\#\{\sigma\in \Sigma_N: H_N(\sigma)\in (E_{\max}-\varepsilon, E_{\max}+\varepsilon)\}\big]\ .
$$
Proceeding as in \eqref{eqn: log number}, we obtain
$$
\frac{1}{N}\log \prob((B_{\varepsilon,\rho}^{\max})^c|\vec{h})\leq -\rho + \log 2 
+  \frac{1}{N}\log \mu_N\left( e^{-\frac{N(E_{\max}-y_{N,\vec{h}}(\sigma))^2}{2}} 
\int_{-\varepsilon \sqrt{N}}^{ \varepsilon \sqrt{N}} \frac{e^{-u^2/2} e^{-\sqrt{N}u(E_{\max}-y_{N,\vec{h}}(\sigma))}}{\sqrt{2\pi}} du\right)
$$
The integral term is smaller than $\exp N \varepsilon(E_{\max } + \frac{1}{N}\sum_i |h_i|)$. 
We apply \eqref{eqn: varadhan} and use the fact that $S(E_{\max})=0$ to get
$$
\lim_{N\to\infty}\frac{1}{N}\log \prob((B_{\varepsilon,\rho}^{\max})^c|\vec{h})\leq -\rho + \varepsilon (E_{\max}+\E|h_1|) \ \text{ $\prob$-a.s.}
$$
We now pick $\rho$ and $\varepsilon$ so that the right-hand side is negative. 

Finally, divide the interval $(E_{\min}+\varepsilon, E_{\max}-\varepsilon)$ in sub-intervals of width $1/\sqrt{N}$ of the form $(E_k, E_k+1/\sqrt{N}]$,
where $k=1, \dots K_N$. Note that $K_N$ is of the order of $\sqrt{N}$. 
Consider the event
$$
C_{k,\rho}=\big\{ e^{N (S(E_k) -\rho)}\leq \mathcal N_N(E_k) \leq e^{N (S(E_k) +\rho)}\big\}
$$
By Proposition \ref{prop: entropy}, we get $\prob \left( \bigcap_{k=1}^{K_N} C_{k,\rho}| \vec{h}\right)$ tends to $1$ exponentially fast. 

By the above, we can restrict the convergence of the probability of $\{f_N(\beta)>f(\beta)+\delta\}$ on the intersection of the events $A_\varepsilon$, $B_{\varepsilon,\rho}^{\max}$, $B_{\varepsilon,\rho}^{\min}$ and  $\bigcap_{k} C_{k,\rho}$. On this event we have that
$$
\sum_{\sigma\in \Sigma_N}\exp \beta H_N(\sigma)\leq \sum_{k=1}^{K_N}e^{ \beta N(E_k+1/\sqrt{N})}e^{N(S(E_k)+\rho)} +e^{\rho N }e^{\beta N (E_{\max} +\varepsilon)}+ e^{\rho N }e^{\beta N E_{\min} }\ .
$$
Therefore, using the notation $E_0=E_{\max}$,
$$
\frac{1}{N} \log \sum_{\sigma\in \Sigma_N}\exp \beta H_N(\sigma)\leq \max_{k=0,\dots K_N} \big\{ \beta E_k + S(E_k) \big\}
+ \rho+\varepsilon +o(1)
$$
The right-hand side tends to $ f(\beta)+ \rho+\varepsilon$ as $N\to\infty$ by continuity.
Thus it suffices to pick $\rho+\varepsilon<\delta$ to get $f_N(\beta)\leq f(\beta) +\delta$ on the considered event.

\end{proof}

Before finishing the proof of Theorem \ref{thm: free}, we need some notation.
For $E\in [E_{\min}, E_{\max}]$, let $y^*=y^*(E)$ be the unique maximizer in the definition of $S(E)$
\begin{equation}
\label{eqn: y*}
y^*(E): =\text{argmax}_{y\in [-\E|h_1|,\E|h_1|]} S(E)\ .
\end{equation}
We also consider its Legendre conjugate $t^*=t^*(E)$
\begin{equation}
\label{eqn: t*}
\psi'(t^*)=\E[h_1\tanh t^* h_1]=y^*\ .
\end{equation}
Note that $0<\psi'(t)< \E|h_1|$ for all $t\in\R$ and that $t^*$ is uniquely defined since $\psi$ is strictly convex. 
Moreover, by Legendre duality, 
\begin{equation}
\label{eqn: dual}
I'(y^*)=t^* \qquad \psi(t^*)+I(y^*)=t^*y^*\ .
\end{equation}
The maximizer is characterized by the stationary conditions
\begin{equation}
\label{eqn: stat}
0= -(E-y^*)-I'(y^*) \Longleftrightarrow E=t^*+ y^*\ .
\end{equation}
In particular, this gives the partition of energy at level $E$: the random field energy is $y^*(E)$ whereas the REM energy density is $t^*(E)$.
Since $S(E_{\max})=0$, we get the following representation for $E_{\max}$
\begin{equation}
\label{eqn: E max}
E_{\max}= \sqrt{2(\log 2 - I(y^*(E_{\max})))} + y^*(E_{\max})\ .
\end{equation}

\begin{proof}[Proof of Theorem \ref{thm: free}]
We use the representation of Proposition \ref{prop: Gibbs}.
Note first that
$$
S'(E)= \frac{\partial}{\partial E} S(E,y)\Big|_{y=y^*(E)}+ \frac{\partial}{\partial y} S(E,y)\Big|_{y=y^*(E)}= 
\frac{\partial}{\partial E} S(E,y)\Big|_{y=y^*(E)}= - E+ y^*(E). 
$$
Therefore for a given $\beta$, maximizers $E^*=E^*(\beta)$ of $\max_{E\in [E_{\min}, E_{\max}]} \big\{ \beta E - S(E)\big\}$ that lie in the open interval are characterized by the equation
\begin{equation}
\label{eqn: max}
\beta + S'(E^*)= 0 \Longleftrightarrow E^*(\beta)= \beta + y^*(E^*)\ .
\end{equation}
In particular, from \eqref{eqn: stat}, we get that the REM energy density at $\beta$ is given by 
\begin{equation}
\label{eqn: t*}
t^*(\beta):=t^*(E^*(\beta))=\beta\ . 
\end{equation}
Moreover, by \eqref{eqn: t*}, the random field energy density is
\begin{equation}
\label{eqn: y*}
y^*(\beta):= y^*(E^*(\beta))= \E[h_1\tanh \beta h_1]\ .
\end{equation}
Note that $y^*(\beta)$ is an increasing function of $\beta$. In particular, \eqref{eqn: max} determines the maximizer $E^*(\beta)$ uniquely when it lies in $(E_{\min}, E_{\max})$. We conclude that whenever $E^*(\beta)<E_{\max}$, the free energy is given by
\begin{equation}
\label{eqn: high}
f(\beta)= \beta(\beta + y^*(\beta) ) + \log 2 -\frac{\beta^2}{2} - I(y^*(\beta))= \frac{\beta^2}{2} +\log 2 +\E[\log\cosh \beta h_1]
\end{equation}
where we used the duality relation \eqref{eqn: dual}. 
For $E^*(\beta)=E_{\max}$, since $S(E_{\max})=0$, we simply have
$
f(\beta)=\beta E_{\max}
$.
It remains to characterize $\beta_c$. By \eqref{eqn: stat}, we must have
\begin{equation}
 E_{\max}=y^*(E_{\max}) + \beta_c\ .
\end{equation}
Since $S(E_{\max})=0$ and $S(E_{\max})=\log 2 -\beta_c^2/2 - I(y^*(E_{\max}))$, we get $\beta_c=\sqrt{2(\log 2- I(y^*(\beta_c))}$. Equation \eqref{eqn: beta c} follows from the duality relation \eqref{eqn: dual}.  
\end{proof}

\section{Proof of Theorem \ref{thm: extremal}}
A finer control of  the fluctuations of the order parameters around their typical values at finite $N$ is needed to prove the theorem on the extremal process.

Let  $\psi_N(t):=\frac{1}{N}\log \mu_N(e^{tN y_{N,\vec{h}}(\sigma)})$ and $I_N(l)=:\max_{t\in\R}\{ tl - \psi_N(t)\}$ the Legendre transform of $\psi_N$.
We have dropped the dependence on $\vec h$ in the notation for simplicity. 
It is easily checked that $\psi_N(t)$ is differentiable and strictly convex for all $t$, hence so is $I_N$. Moreover, 
\begin{equation}
\label{eqn: psi deriv}
\psi'_N(t)= \frac{1}{N}\sum_{i=1}^N h_i \tanh t h_i \qquad \psi_N''(t)=\frac{1}{N}\sum_{i=1}^N  \frac{h_i^2}{\cosh^2 t h_i}\ .
\end{equation}
Note that $\psi'_N(t)\to \psi'(t)$ and $\psi''_N(t)\to \psi''(t)$ $\prob$-a.s. for all $t\in\R$.

For a given $E\in \R$, consider
\begin{equation}
\label{eqn: S_N}
S_N(E ):=\max_{-\E|h_1|\leq y\leq \E|h_1|}\left\{\log 2- \frac{1}{2}\big(E-y\big)^2- I_N(y)\right\}
\end{equation}
which is roughly the log-number of configurations with energy density $E$ at finite $N$.
Define 
\begin{equation}
\label{eqn: c_1}
c_1(N,\vec{h}):=\sup_{E\in\R}\big\{ E\in\R: S_N(E)>0\big\}\ .
\end{equation}
Note that $c_1(N,\vec{h})$ must exist and is unique. 
Define $y_N^*$ to be the unique maximizer of $S_N(c_1(N,\vec{h}))$. Write $t_N^*$ for its conjugate that is 
\begin{equation}
\psi'_n(t^*_N)=y_N^* \qquad I_N(y_N^*)=y_N^* t^*_N - \psi_N(t^*_N) \ .
\end{equation}
Since $y_N^*$ is the unique maximizer it must satisfy the stationary condition:
\begin{equation}
\label{eqn: c_1 stat}
(c_1(N,\vec{h})-y_N^*)- I'_N(y_N^*)=0 \Longrightarrow c_1(N,\vec{h})=y_N^*+ t_N^*\ .
\end{equation}
Moreover, from the definition of $S_N$, we have
\begin{equation}
\label{eqn: c_1 conv}
c_1(N,\vec{h})= \sqrt{2(\log 2 - I_N(y_N^*))} + y_N^*\ .
\end{equation}
Finally, define $c_2(N,\vec{h})$ as
\begin{equation}
\label{eqn: c_2}
c_2(N,\vec{h}):= \frac{1}{2(c_1(N,\vec{h}) - y_N^*)}\ .
\end{equation}

The next lemma establishes the convergence of $c_1(N, \vec{h})$ and of $c_2(N,\vec{h})$ to deterministic limits.
The first assertions are standard LDP results and are included for completeness. 
\begin{lem}
\label{lem: conv}
For every $y\in (-\E|h_1|, \E|h_1|)$, we have 
$$
\lim_{N\to\infty} I_N(y)=I(y) \ \text{ $\prob$-a.s.}
$$
where $I$ is defined in \eqref{eqn: I}. Moreover,
$$
\lim_{N\to\infty} y_N^*= y^* \qquad \lim_{N\to\infty}t_N^* = t^*  \ \text{ $\prob$-a.s.}
$$
where $y^*$ and $t^*$ are defined in \eqref{eqn: y*} and \eqref{eqn: t*}; and
$$
\lim_{N\to\infty}c_1(N,\vec{h})= E_{\max} \qquad \lim_{N\to\infty} c_2(N,\vec{h})=\frac{1}{2\beta_c}\   \ \text{ $\prob$-a.s.}
$$
\end{lem}
\begin{proof}

Let $y\in (-\E|h_1|, \E|h_1|)$. 
By definition, we have 
$$
I_N(y)=t_N y - \psi_N(t_N)
$$
where $t_N$ is such that $\psi'_N(t_N)=y$.
(Note that such a $t_N$ exists for $N$ large enough since $\lim_{t\to\pm \infty}\psi'_N(t)= \frac{\pm 1}{N} \sum_{i=1}^N |h_i|$. )
For $I_N(y)\to I(y)$, we prove that $t_N\to t$ and $\psi_N(t_N)\to\psi(t)$ where $\psi'(t)=y$. 
The second convergence follows from the first.

For $t_N\to t$, we show that every converging subsequence has the same limit. 
First note that if $(t_{N_k})_k$ is a subsequence converging to $\tilde t$, we must have by convexity that
$
\psi'(\tilde t)=y
$
since $\psi_N(t)\to\psi(t)$ and $\psi(t)$ is differentiable everywhere.  
In particular, the limit must be unique since $\psi$ is strictly convex and we must have $\tilde t=t$. 
It remains to show that $\limsup_{N\to\infty} t_N<\infty$ and $\liminf_{N\to\infty} t_N>-\infty$.
We show the former, the latter being similar. Suppose there exists $(t_{N_k})_k$ such that $t_{N_k}\to+\infty$. 
We have for any $s<t_{N_k}$
$$
\psi'_{N_k}(t_{N_k})= \psi'_{N_k}(s) + (t_{N_k} - s) \psi''(\bar t_{N_k})> \psi'_{N_k}(s)
$$
where $ s\leq \bar t_{N_k}\leq t_{N_k}$, and since $\psi''>0$. But $\lim_{s\to +\infty} \lim_{k\to\infty}  \psi'_{N_k}(s)= \E|h_1|$.
This is a contradiction since $\psi'_{N_k}(t_{N_k})=y$ for all $k$ by definition. This proves $I_N(y)\to I(y)$ for all $y\in (-\E|h_1|, \E|h_1|)$.

Observe that
\begin{equation}
\label{eqn: I*}
\text{ if $y_N^*\to y^*$ then $I_N(y_N^*)\to I(y^*)$.}
\end{equation}
 Indeed,
$$
I_N(y_N^*)= I_N(y^*) + (y_N^*-y^*) I'_N(\bar y_N^*)\ ,
$$
where $\bar y_N^*$ is between $y_N^*$ and $y^*$. Since $I$ is convex and differentiable on $[-\E|h_1|,\E[|h_1|]$, we must have
$I'_N(\bar y_N^*)\to I'(y^*)$. Equation \eqref{eqn: I*} then follows from the convergence of $I_N(y^*)$ proved before.
We now prove $y_N^*\to y^*$. 
By definition, the sequence $(y_N^*)_N$ is in the compact interval $[-\E|h_1|, \E|h_1|]$. 
Let $(y_{N_k}^*)_k$ be a converging subsequence and $\tilde y^*$ its limit. 
By definition,  the following relation must be satisfied for all $k$
 $$
 0= \log 2 -\frac{1}{2}(c_1(N_k) - y_{N_k}^*)^2- I_{N_k}(y_{N_k}^*)
$$
In the limit $k\to\infty$, we recover the relation \eqref{eqn: stat}, which defines $y^*$ uniquely.
The convergence $t_N^*\to t^*$ is done exactly as the convergence of $t_N$ in the first part of the proof and we omit the details.

The convergence $c_1(N,\vec{h})\to E_{\max}$ follows directly from \eqref{eqn: c_1 conv} and  \eqref{eqn: I*}. 
The convergence of $c_2(N,\vec{h})$ is direct from the one of $c_1(N,\vec{h})$. 
\end{proof}

\begin{proof}[Proof of Theorem \ref{thm: extremal}]
For conciseness, we will omit the dependence on $\vec{h}$ throughout the proof and write
$y_{N,\vec{h}}(\sigma)=y_N(\sigma)$, $c_1(N,\vec {h})=c_1$, $c_2(N,\vec{h})=c_2$ for conciseness.
For a continuous function of compact support $\phi: \R\to [0,\infty)$, we show
\begin{equation}
\label{eqn: to prove}
\E\left[ \exp(-\sum_{\sigma\in\Sigma_N} \phi(H_N(\sigma)-r(N,\vec{h})) \Big| \vec h\right] \to \exp\left(-\int_\R(1-e^{-\phi(z)}) C  e^{-\beta_c z} dz\right)\ \text{$\prob$-a.s.}
\end{equation}
Since the $X_N(\sigma)$'s are independent Gaussians of variance $N$, the left-hand side is equal to 
\begin{equation}
\label{eqn: int}
\prod_{\sigma\in \Sigma^N}  \left(1- \int_\R (1-e^{-\phi(z)}) \ \frac{\exp\left(\frac{-1}{2N}(z+r_N - Ny_N(\sigma))^2\right)}{\sqrt{2\pi N}}\ dz\right)\ .
\end{equation}
We develop the square:
\begin{equation}
\label{eqn: square}
\begin{aligned}
&\frac{1}{2N}\Big(z+c_1N - c_2 \log N - N y_{N,\vec{h}}(\sigma)\Big)^2\\
&=  \frac{N}{2}\big(c_1 - y_{N,\vec{h}}(\sigma)\big)^2   + \big(-c_1c_2 + c_2 y_{N,\vec{h}}(\sigma)\big)\log N + \big(c_1 -  y_{N,\vec{h}}(\sigma) \big)z + o(1)\ .
\end{aligned}
\end{equation}
Therefore
\begin{equation}
\label{eqn: square2}
\exp\left(\frac{-1}{2}(y+r_N - Ny_{N,\vec{h}}(\sigma))^2\right)=
(1+o(1)) \ e^{- \frac{N}{2}\big(c_1-y_{N}(\sigma)\big)^2} e^{- (c_1-y_{N}(\sigma))z }\  N^{c_1c_2 - c_2 y_N(\sigma)} \ .
\end{equation}
Putting this back in the integral of \eqref{eqn: int} gives for each $\sigma$
$$
(1+o(1))  \frac{1}{\sqrt{2\pi}}\int (1-e^{-\phi(z)})   \left(\frac{e^{- \frac{N}{2}\big(c_1-y_N(\sigma)\big)^2}}{N^{-c_1c_2 + c_2 y_N(\sigma) +1/2} } \ e^{- (c_1-y_N(\sigma))z}\ \right) dz\ .
$$

To prove \eqref{eqn: to prove}, it remains to show that for every $z \in \R$,  
\begin{equation}
\label{eqn: sum}
\frac{e^{-c_1z}}{\sqrt{2\pi}} \sum_{\sigma\in \Sigma^N}  \frac{e^{- \frac{N}{2}\big(c_1-y_N(\sigma)\big)^2}}{N^{-c_1c_2 + c_2 y_N(\sigma) +1/2}  }e^{-y_N(\sigma)) z} \to C e^{-\beta_c z} \ .
\end{equation}
This is reminiscent of \eqref{eqn: log number}. However, the sum has to be control at the finer scales of the central limit theorem as opposed to large deviations.
The sum \eqref{eqn: sum} is 
\begin{equation}
\label{eqn: sum1}
\mu_N\left(\frac{\exp\left(N\left\{\log 2- \frac{1}{2}\big(c_1-y_N(\sigma)\big)^2\right\}\right)}{N^{-c_1c_2 + c_2 y_N(\sigma) +1/2}} e^{y_N(\sigma) z} \right) 
\end{equation}
For an arbitrary $y$, we take $y_N(\sigma)=\big(y_N(\sigma)-y\big)+y$ to write the integrand as
\begin{equation}
\label{eqn: sum2}
\frac{\exp N \left(\left\{\log 2- \frac{1}{2}\big(c_1-y\big)^2\right\}  - \frac{1}{2}(y_N(\sigma)-y)^2  + (c_1-y_N(\sigma))(y_N(\sigma)-y)\right)}
{N^{-c_1c_2 + c_2 y +1/2} N^{c_2(y_N(\sigma)-y) } }e^{\big(y + (y_N(\sigma)-y)\big)z} 
\end{equation}
We introduce $I_N(y)$ in the exponential and take $y=y_{N}^*$ which is the value of $y$ for which the exponential term is maximal.
(We stress that choosing $y=y^*$ would not yield the convergence. Optimization at every finite $N$ is necessary.)
Moreover, by the choice of $c_1$ in \eqref{eqn: c_1}, we have that this maximum is $S_N(c_1)=0$. Finally, by the choice of $c_2$ in \eqref{eqn: c_2}, the first term in the denominator vanishes 
leaving:
\begin{equation}
\label{eqn: sum3}
\exp \left\{N\left(\frac{-1}{2}(y_N(\sigma)-y_N^*)^2  + t_N^*(y_N(\sigma)-y_N^*) + t_N^*y_N^* -\psi(t_N^*)\right)\right\}
\frac{e^{ (y_N(\sigma)-y_N^*)z} e^{y_{N}^*z} }{N^{c_2(y_N(\sigma)-y_N^*) } } 
\end{equation}
where we have used the stationary condition \eqref{eqn: c_1 stat} and the fact that $I_N(y_N^*)=t_N^*y_N^*-\psi_N(t_N^*)$.
Consider the tilted measure on $\Sigma_N$:
\begin{equation}
\label{eqn: tilt}
\tilde\mu_N(d\sigma)= \frac{e^{tN y_N(\sigma)} }{\mu_N(e^{tN y_N(\sigma)})} \mu_N(d\sigma)\ .
\end{equation}
Taking $\tilde \mu_N$  for $t=t^*_N$, we have under this measure
$$
\tilde \mu_N(y_N(\sigma))= \psi'_N(t_N^*)=y_{N}^*
\qquad \text{Var}_{\tilde \mu_N}(y_N(\sigma))=N\psi_N''(t_N^*)\ .
$$ 
Thus,   with the notation $\bar y_N(\sigma):=\sqrt{N}(y_N(\sigma)-y_N^*)$, the expectation of \eqref{eqn: sum3} under $\mu_N$ is
\begin{equation}
\label{eqn: sum5}
e^{y_N^*z} \ \tilde \mu_N\left(e^{ \frac{-1}{2}\bar y^2_N(\sigma)} \frac{ e^{\frac{\bar y_N(\sigma)}{\sqrt{N}} z} }{N^{ c_2 \frac{\bar y_N(\sigma)}{\sqrt{N}}}}\right)
\end{equation}
It remains to prove a central limit theorem for $y_N(\sigma)$. By construction
$$
\tilde \mu_N(\bar y_N)=0 \qquad \text{Var}_{\tilde \mu_N}(\bar y_N)= \psi''(t_N^*)\ .
$$
\begin{lem}
\label{lem: CLT}
For $\vec{h}$ on a set of $\prob$-probability one, the  random variables $(\bar y_{N,\vec{h}}(\sigma))_{N\in \N}$ on the probability space $(\Sigma_N, \tilde \mu_N)$ with $t=t_N^*$ converges in law
to a Gaussian variable of mean $0$ and variance $\psi''(\beta_c)$. 
\end{lem}
\begin{proof}
Since the random variables $(\bar y_N)_{N\in \N}$ are IID under $\tilde \mu_N$, the result follows from the Lindeberg-Feller CLT theorem (see e.g. Theorem 3.4.5 in \cite{durrett}) if $\text{Var}_{\tilde \mu_N}(\bar y_{N})= \psi''(t_N^*)$ converges to $\psi''(\beta_c)$. But the convergence holds by continuity of $\psi''$ and the convergence of $t_N^*\to t^*=\beta_c$ in Lemma \ref{lem: conv}.
\end{proof}
Clearly, $y\mapsto e^{-y^2/2}$ is a bounded continuous function. Hence, by Lemma \ref{lem: CLT}, 
$$
\mu_N\left(e^{ \frac{-1}{2}\bar y^2_N(\sigma)}\right)
\to
\int_\R e^{ \frac{-1}{2} y^2} \ \frac{e^{\frac{-y^2}{2\psi''(\beta_c)}}}{\sqrt{2\pi \psi''(\beta_c)}} dy = \frac{1}{\sqrt{1+\psi''(\beta_c)}}\ .
$$
The fraction term in \eqref{eqn: sum5} goes to $1$ by Lemma \ref{lem: CLT}. Putting all this together in \eqref{eqn: sum1} gives
\begin{equation}
\frac{1}{\sqrt{2\pi}} \sum_{\sigma\in \Sigma_N}  \frac{e^{- \frac{N}{2}\big(c_1-y_{N}(\sigma)\big)^2}}{N^{-c_1c_2 + c_2 y_{N}(\sigma) +1/2}} e^{-(c_1-y_{N}(\sigma)) z} 
\to \frac{ e^{-\beta_c z}}{\sqrt{2\pi(1+\psi''(\beta_c))}}
\end{equation}
where we used Lemma \ref{lem: conv} to get the convergence of $c_1$ to $E_{\max}$.  
\end{proof}

\section{The Gibbs measure and the overlap distribution}


We recall that a point process $\xi=(\xi_i, i\in\N)$ with $\xi_1>\xi_2>\dots>0$ with $\sum_i\xi_1=1$ 
is a Poisson-Dirichlet variable with parameter $0<\beta_c/\beta<1$ if it has the same law as 
\begin{equation}
\left(\frac{\exp \beta\eta_i}{\sum_j \exp \beta\eta_j}, i\in \N \right)_\downarrow
\end{equation}
where $\eta$ is a Poisson process with intensity $Ce^{-\beta_c y}dy$ on $(0,\infty)$.
Therefore, in view of Theorem \ref{thm: extremal}, the proof of Theorem \ref{thm: gibbs} is reduced to show that the normalization of the weights is a continuous procedure under the weak convergence. 
Similar arguments have been used in Chapter 1 of \cite{talagrand} and in \cite{bolthausen-kistler} for other REM-related models.

\begin{proof}[Proof of Theorem \ref{thm: gibbs}]
Throughout the proof, we will often drop the dependence on $\vec{h}$ for simplicity.
Convergence of the point process in Theorem \ref{thm: extremal} is limited to test-functions with compact support. 
We thus have to limit the normalization of the weights to a compact set.
For $\delta>0$, we consider the truncated partition function
$$
Z_{N}^\delta(\beta):= \sum_{\sigma\in \Sigma_N} \exp \beta H_N(\sigma) 1_{\{\sigma: H_N(\sigma)-r_N \in [-\delta, \delta]\}}
$$
and the corresponding truncated Gibbs weights
\begin{equation}
\label{eqn: cutoff}
G_{\beta,N}^\delta(\sigma):=
\begin{cases}
\frac{\exp\beta H_N(\sigma)}{Z^\delta_{N}(\beta)}\ &\text{if $H_N(\sigma)-r_N\in[-\delta,\delta]$}\\
0  \ &\text{otherwise.}
\end{cases}
\end{equation}
The analogous truncation can be done for $\xi$
$$
\xi_i^\delta=
\begin{cases}
\frac{\exp \beta\eta_i}{\sum_j \exp\beta\eta_j 1_{\{\eta_j\in [-\delta, \delta]\}}}\ &\text{if $\eta_i\in[-\delta,\delta]$}\\
0  \ &\text{otherwise.}
\end{cases}
$$

Let $f$ be a continuous function on the compact space of mass partitions $\mathcal S=\{ \vec{p}=(p_i,i\in\N): 1\geq p_1 \geq p_2 \geq \dots\geq 0, \sum_i p_i\leq 1\}$
equipped with the metric $d(\vec{p},\vec{p'}):=\sum_{n\geq 1}2^{-n}|p_n-p_n'|$. 
We will show that for $G_{\beta,N}:=(G_{\beta,N}(\sigma),\sigma\in \Sigma_N)_\downarrow$ and $\xi$ the above Poisson-Dirichlet variable, we have
$$
\Big| \E[f(G_{\beta,N})]-\E[f(\xi)]\Big|\to 0 \qquad \text{as $N\to\infty$.}
$$
For the truncated weights $G_{\beta,N}^\delta:=(G_{\beta,N}(\sigma),\sigma\in \Sigma_N)_\downarrow$ and $\xi_\delta=(\xi_i^\delta, i\in\N)_\downarrow$, we have
\begin{equation}
\label{eqn: approx}
\Big| \E[f(G_{\beta,N})]-\E[f(\xi)]\Big|\leq \Big| \E[f(G_{\beta,N})]-\E[f(G_{\beta,N}^\delta)]\Big| + \Big| \E[f(G_{\beta,N}^\delta)]-\E[f(\xi^\delta)]\Big| + \Big| \E[f(\xi^\delta)]-\E[f(\xi)]\Big|\ .
\end{equation}
The second term converges to $0$ by Theorem \ref{thm: extremal} for a fixed $\delta$. 
The first and third terms are handled similarly. We detailed the proof for the first term.
Since $f$ is continuous, it suffices to show that, with large probability, $d(G_{\beta,N}^\delta, G_{\beta,N})$ is small uniformly in $N$ for a large but fixed $\delta$. Elementary manipulations give
$$
d(G_{\beta,N}^\delta, G_{\beta,N})\leq \sum_{\sigma}|G_{\beta,N}^\delta(\sigma)-G_{\beta,N}(\sigma)|\leq 2\frac{Z-Z_N(\beta)}{Z_{N}(\beta)}\ .
$$
It remains to show that we can pick $\delta$ such that for every $\varepsilon>0$ 
$$
\prob\left(\left|\frac{Z_N(\beta)-Z^\delta_N(\beta)}{Z_{N}(\beta)}\right| >\varepsilon\right)<\varepsilon\ , \text{uniformly in $N$.}
$$
We note that we can pick $\varepsilon'>0$ small enough such that $\prob(Z_N(\beta)\leq \varepsilon')<\epsilon$. This is because
$\prob(Z_N(\beta)\leq \varepsilon')\leq \prob(\#\{\sigma: H_N(\sigma)-r_N\leq \frac{1}{\beta}\log \varepsilon' \}=0)$. 
By Theorem \ref{thm: extremal}, the latter probability can be made smaller than $\varepsilon$ uniformly in $N$. 
The same way, we can choose $\delta$ large enough such that $\prob(\max_\sigma H_N(\sigma)-r_N>\delta)<\varepsilon$.
Putting all this together, it remains to estimate $\prob(Z_N(\beta)-Z^\delta_N(\beta)>\varepsilon\varepsilon', \max_\sigma H_N(\sigma)-r_N\leq\delta)$.
Since there is no point above $\delta$ on this event, this is bounded above by the Markov's inequality
$$
\sum_{\sigma\in \Sigma_N}\int_{-\infty}^{-\delta} e^{\beta x} \ \prob(H_N(\sigma)-r_N\in dx)\ .
$$
By Theorem \ref{thm: extremal}, this converges to $\int_{-\infty}^{-\delta} e^{\beta z} C e^{-\beta_c z}dz$. This can be made arbitrarily small by taking $\delta$ large because $\beta>\beta_c$.
This concludes the proof of the theorem.
\end{proof}

\begin{proof}[Proof of Corollary \ref{cor: overlap}]
The function on the space $\mathcal S$ of mass partitions defined by $\vec p\mapsto \sum_i p_i^2$ is continuous under the metric. 
Therefore, it follows from Theorem \ref{thm: gibbs} that
$$
\E G_{\beta, N}^{\times 2}\{R_N(\sigma,\sigma')=1\}= \E\left[ \sum_{\sigma\in \Sigma_n} \big(G_{\beta,N}(\sigma)\big)^2\right]\to \E\left[\sum_i \xi^2\right]=1-\frac{\beta_c}{\beta}
$$
where $\xi$ is Poisson-Dirichlet with parameter $\beta_c/\beta$.  The last equality is a standard computation, see for example \cite{ruelle}.

It remains to show that two distinct $\sigma$ and $\sigma'$ with extremal energies must have overlap $q$.
Consider for $\varepsilon>0$ the subset $I_\varepsilon:=[0,q-\varepsilon]\cup[q+\varepsilon,1)$. 
Let $\delta>0$.
We will prove that
\begin{equation}
\label{eqn: pairs}
\E[\#\{(\sigma,\sigma'): H_N(\sigma), H_N(\sigma')\in [-\delta+r_N,\delta+r_N] \text{ and } R_N(\sigma,\sigma') \in I_\varepsilon\}]\to 0 \ \text{ as $N\to\infty$.}
\end{equation}
Recall the truncation introduced in \eqref{eqn: cutoff}. Note that $\E {G^{\delta\ \times 2}_{\beta,N}}\{R_N(\sigma,\sigma') \in I_\varepsilon\}$ is smaller
than the left-hand side of \eqref{eqn: pairs}. The same approximation as in \eqref{eqn: approx} will then yield the result
$$
\E G_{\beta,N}^{\times 2}\{R_N(\sigma,\sigma') \in I_\varepsilon\} \to 0\ .
$$
The same manipulations done from \eqref{eqn: square} to \eqref{eqn: sum5} to re-express the Gaussian densities can be applied {\it verbatim} to the right-hand side of \eqref{eqn: pairs} and yield
$$
\frac{1}{2\pi}\int_{[-\delta,\delta]^{\times 2}} e^{-\beta_c(z+z')}
\tilde\mu_N\times \tilde\mu_N\left( e^{-\frac{1}{2}(\bar y^2_N(\sigma)+\bar y^2_N(\sigma'))} \frac{ e^{\frac{\bar y_N(\sigma)+\bar y_N(\sigma')}{\sqrt{N}} z} }{N^{ c_2 \frac{\bar y_N(\sigma)+\bar y_N(\sigma')}{\sqrt{N}}}}
\ 1_{\{R_N(\sigma,\sigma')\in I_\varepsilon\}}
\right)dz\ dz'
$$
Therefore, convergence to zero would follow if we prove a weak law of large number for $R_N(\sigma, \sigma')$ with mean $q$ under $\tilde\mu_N\times \tilde\mu_N$.
This is straightforward from the fact that the variables $(\sigma_i\sigma_i', i\leq N)$ are independent under $\tilde\mu_N\times \tilde\mu_N$ with mean
$$
\tilde\mu_N\times \tilde\mu_N(\sigma_i\sigma_i')=\tilde\mu_N(\sigma_i)^2=\left(\frac{1}{t_N^*}\partial_{h_i}\psi_N(t_N^*)\right)^2=\tanh^2 t_N^*h_i\ .
$$
The second equality follows from \eqref{eqn: tilt} and the third is similar to \eqref{eqn: psi deriv}. This implies that $\tilde\mu_N\times \tilde\mu_N(R_N(\sigma,\sigma'))$
converges to $\E[\tanh^2 \beta_ch_1]$ by Lemma \ref{lem: conv}.

\end{proof}

\end{document}